\documentclass[a4paper]{amsart}
\usepackage{graphicx}
\usepackage{amssymb}
\usepackage{amsmath}
\usepackage{amsthm}
\usepackage{amscd}
\usepackage[all,2cell]{xy}

\UseAllTwocells \SilentMatrices
\newtheorem{thm}{Theorem}[section]

\newtheorem{cor}[thm]{Corollary}
\newtheorem{lem}[thm]{Lemma}
\newtheorem{exm}[thm]{Example}

\newtheorem{prop}[thm]{Proposition}
\theoremstyle{definition}
\newtheorem{defn}[thm]{Definition}
\theoremstyle{remark}

\numberwithin{equation}{section}

\begin{document}
\title[The Auslander bijections and universal extensions]
{The Auslander bijections and universal extensions}

\author[Xiao-Wu Chen] {Xiao-Wu Chen}

\subjclass[2010]{16G70, 16G10, 13D07}
\date{the revised version, \today}


\keywords{universal extension, determined morphism, Auslander-Reiten duality, Auslander bijection}%

\maketitle

\dedicatory{}%
\commby{}%

\begin{abstract}
Universal extensions arise naturally in the Auslander bijections. For an abelian category having Auslander-Reiten duality, we exploit a bijection triangle, which involves the Auslander bijections, universal extensions and the Auslander-Reiten duality. Some consequences are given, in particular, a conjecture by Ringel is verified.
\end{abstract}

\section{Introduction}

The theory of morphisms determined by objects  is initiated in \cite{Aus1, Aus2}, which puts the Auslander-Reiten theory in a more general context. The Auslander bijections, termed in \cite{Rin2}, play a central role in this general theory. Roughly speaking, the Auslander bijections reduce the study of morphisms to submodules, where the latter has a  geometric feature via the Grassmannians of submodules; see \cite{Rin3}.

In what follows, we explain the setting of the Auslander bijections. Let $\mathcal{A}$ be an additive category, which is assumed to be skeletally small.

 Let $C$ and $Y$ be two objects in $\mathcal{A}$. A morphism $\alpha\colon X\rightarrow Y$ is \emph{right $C$-determined}  if the following holds: if $t\colon T\rightarrow Y$ is any morphism satisfying that $t\circ \phi$ factors through $\alpha$ for each morphism $\phi\colon C\rightarrow T$, then $t$ factors through $\alpha$.

 Two morphisms $\alpha_1\colon X_1\rightarrow Y$ and $\alpha_2\colon X_2\rightarrow Y$ are \emph{right equivalent} if $\alpha_1$ factors through $\alpha_2$ and $\alpha_2$ factors through $\alpha_1$. This is an equivalence relation on the set of all morphisms to $Y$, and the right equivalence class of a morphism $\alpha\colon X\rightarrow Y$ is denoted by $[\alpha\rangle$. Then  we have the poset $[\longrightarrow Y\rangle$ of right equivalence classes of morphisms to $Y$. Here, ${[\alpha\rangle} \leq  {[\alpha'\rangle}$ provided that $\alpha$ factors through $\alpha'$.

  We observe that if $\alpha$ and $\alpha'$ are right equivalent, then $\alpha$ is right $C$-determined if and only if so is $\alpha'$; so it makes sense to say that the corresponding right equivalence class is right $C$-determined. Following \cite{Rin2}, we denote by $^C[\longrightarrow Y\rangle$ the subset of $[\longrightarrow Y\rangle$ consisting of right equivalence classes that are right $C$-determined. This poset is central in this theory; see \cite{Rin2}.

 Set $\Gamma(C)={\rm End}_\mathcal{A}(C)$ to be the endomorphism ring of $C$, and denote by $\Gamma(C)^{\rm op}$ its opposite ring. A morphism $\alpha\colon X\rightarrow Y$ in $\mathcal{A}$ induces a morphism of left $\Gamma(C)^{\rm op}$-modules ${\rm Hom}_\mathcal{A}(C, \alpha)\colon {\rm Hom}_\mathcal{A}(C, X)\rightarrow {\rm Hom}_\mathcal{A}(C, Y)$. Its image ${\rm Im}\; {\rm Hom}_\mathcal{A}(C, \alpha)$ is a $\Gamma(C)^{\rm op}$-submodule of ${\rm Hom}_\mathcal{A}(C, Y)$. Observe that if $\alpha_1$ and $\alpha_2$ are right equivalent, then ${\rm Im}\; {\rm Hom}_\mathcal{A}(C, \alpha_1)={\rm Im}\; {\rm Hom}_\mathcal{A}(C, \alpha_2)$.

 We denote by ${\rm Sub}_{\Gamma(C)^{\rm op}}\; {\rm Hom}_\mathcal{A}(C, Y)$ the poset formed by $\Gamma(C)^{\rm op}$-submodules of ${\rm Hom}_\mathcal{A}(C, Y)$, ordered by  the inclusion. Then the following map is well defined
 \begin{align*}
 \eta_{C, Y}\colon [\longrightarrow Y\rangle \longrightarrow {\rm Sub}_{\Gamma(C)^{\rm op}}\; {\rm Hom}_\mathcal{A}(C, Y), \quad [\alpha\rangle \mapsto {\rm Im}\; {\rm Hom}_\mathcal{A}(C, \alpha).
 \end{align*}

 The restriction of $\eta_{C, Y}$ on $^C[\longrightarrow Y\rangle$ is injective and reflects the orders, that is, for two classes $[\alpha\rangle$ and $[\alpha'\rangle$ that are right $C$-determined, $[\alpha\rangle \leq  [\alpha'\rangle$ if and only if $\eta_{C, Y}([\alpha\rangle)\subseteq \eta_{C, Y}([\alpha'\rangle)$; see \cite[Proposition 4.3]{Rin2}.

We say that  the \emph{Auslander bijection} at $Y$ relative to $C$ holds,  provided that the following map
\begin{align}
 \eta_{C, Y} \colon ^C[\longrightarrow Y\rangle \longrightarrow {\rm Sub}_{\Gamma(C)^{\rm op}}\; {\rm Hom}_\mathcal{A}(C, Y)
 \end{align}
is surjective, or equivalently, it is an isomorphism of posets.

 One of  the fundamental results in the representation theory of artin algebras claims that  if $\mathcal{A}$ is the category of finitely generated modules over an artin algebra, the Auslander bijection at any module holds \cite{ARS, Rin1, Rin2}. This is extended to dualizing varieties over a commutative artinian ring in \cite{Kr}; see also \cite{ChL}.  However, if $\mathcal{A}$ is an abelian category with Serre duality, the Auslander bijection may fail, in which case we have to restrict to epimorphisms; see \cite{ChL}.

Let $k$ be a commutative artinian ring. In what follows, $\mathcal{A}$ is a Hom-finite $k$-linear abelian category having \emph{Auslander-Reiten duality} with  $\tau$ the \emph{Auslander-Reiten translation} \cite{LZ}. The main examples in the representation theory are the category of finitely generated modules  over an artin algebra and the category of coherent sheaves on a weighted projective line.

We denote by $\mathcal{P}$ the ideal of $\mathcal{A}$ formed by \emph{projectively trivial} morphisms, and denote by $\underline{\rm Hom}_\mathcal{A}(C, Y)={\rm Hom}_\mathcal{A}(C, Y)/\mathcal{P}(C, Y)$. Here, we recall that a morphism $f\colon C\rightarrow Y$ is projectively trivial if it factors through any epimorphism ending at $Y$. Any $\Gamma(C)^{\rm op}$-submodule of $\underline{\rm Hom}_\mathcal{A}(C, Y)$ corresponds to a unique $\Gamma(C)^{\rm op}$-submodule of ${\rm Hom}_\mathcal{A}(C, Y)$  that contains $\mathcal{P}(C, Y)$.  Then the poset ${\rm Sub}_{\Gamma(C)^{\rm op}}\; \underline{\rm Hom}_\mathcal{A}(C, Y)$ is viewed as a subset of ${\rm Sub}_{\Gamma(C)^{\rm op}}\; {\rm Hom}_\mathcal{A}(C, Y)$.

 We denote by $[\longrightarrow Y\rangle_{\rm epi}$ the subset of $[\longrightarrow Y\rangle$ formed by epimorphisms $\alpha\colon X\rightarrow Y$. Observe that $\mathcal{P}(C, Y)\subseteq {\rm Im}\; {\rm Hom}_\mathcal{A}(C, \alpha)$. Then we have the following map
 \begin{align*}
 \eta_{C, Y} \colon [\longrightarrow Y\rangle_{\rm epi} \longrightarrow {\rm Sub}_{\Gamma(C)^{\rm op}}\; \underline{\rm Hom}_\mathcal{A}(C, Y), \quad [\alpha\rangle \mapsto {\rm Im}\; {\rm Hom}_\mathcal{A}(C, \alpha)/{\mathcal{P}(C, Y)}.
 \end{align*}

 Set $^C[\longrightarrow Y\rangle_{\rm epi}=[\longrightarrow Y\rangle_{\rm epi}\cap {^C[\longrightarrow Y\rangle}$. We say that the \emph{restricted Auslander bijection} at $Y$ relative to $C$ holds,  provided that the following map
\begin{align}
 \eta_{C, Y} \colon ^C[\longrightarrow Y\rangle_{\rm epi} \longrightarrow {\rm Sub}_{\Gamma(C)^{\rm op}}\; \underline{\rm Hom}_\mathcal{A}(C, Y)
 \end{align}
 is surjective, or equivalently, it is an isomorphism of posets.

The main result claims that a bijection triangle is commutative, which involves the restricted Auslander bijections, universal extensions and the Auslander-Reiten duality. In particular,  the restricted Auslander bijection  always holds in an abelian category having Auslander-Reiten duality.

\vskip 10pt

 \noindent {\bf Theorem.} \emph{Let $\mathcal{A}$ be a Hom-finite $k$-linear abelian category having Auslander-Reiten duality, and let $C, Y$ be objects in $\mathcal{A}$. Set $K=\tau C$. Then the following bijection triangle is commutative
\[\xymatrix @R=1.5pc @C=0.8pc{
&  {{\rm Sub}_{\Gamma(C)^{\rm op}} \; \underline{\rm Hom}_\mathcal{A}(C, Y)}  \\
{^C[\longrightarrow Y\rangle_{\rm epi}}  \ar[ru]^{\eta_{C,Y}}  \ar[rr]^{\delta_{K, Y}}&& {\rm Sub}_{\Gamma(K)}\; {\rm Ext}^1_\mathcal{A}(Y, K) \ar[lu]_{\gamma_{K, Y}},
}\]
where $\gamma_{K, Y}$ is an anti-isomorphism of posets  which is induced by the Auslander-Reiten duality,  and $\delta_{K, Y}$ is an anti-isomorphism of posets whose inverse assigns to each $\Gamma(K)$-submodule $L$ of ${\rm Ext}_\mathcal{A}^1(Y, K)$ the epimorphism in the corresponding universal $L$-extension.}

\vskip 10pt

We mention that the idea of a bijection triangle is implicit in the proof of the Auslander bijections for module categories; see \cite{Aus1, Aus2} and \cite[Chapter X]{ARS}. We are inspired by the comparison between the Auslander bijections and the Auslander-Reiten theory in \cite[Section 10]{Rin2}. Universal extensions of modules are widely used to construct certain modules in  the representation theory of artin algebras. Here, we make them explicit in an arbitrary abelian category.

The paper is structured as follows.  Semi-universal extensions and universal extensions are studied in Sections 2 and 3, where the bijection $\delta_{K, Y}$ is established in Proposition \ref{prop:delta} and the  notion of universal extension is given in Definition \ref{defn:uni-ext}.  The bijection triangle is proved in Section 4; see Theorem \ref{thm:1}. Some consequences of the bijection triangle are given in Section 5. In particular, a conjecture by Ringel in \cite{Rin2} is verified; see Proposition \ref{prop:Ringel} and Corollary \ref{cor:Ringel}.

\section{Semi-universal extensions}

In this section, we study semi-universal extensions in an arbitrary abelian category. We obtain a bijection, which relates the poset of right equivalence classes of morphisms to the poset of finitely generated submodules of the Ext module.

Throughout $\mathcal{A}$ is an abelian category, which is skeletally small. For an object $K$, denote by $\Gamma(K)={\rm End}_\mathcal{A}(K)$ its endomorphism ring. Then for an object $Y$, ${\rm Hom}_\mathcal{A}(Y, K)$ and ${\rm Ext}_\mathcal{A}^1(Y, K)$ carry  natural left $\Gamma(K)$-module structures. For an exact sequence $\xi\colon 0\rightarrow K \rightarrow X\rightarrow Y\rightarrow 0$, we denote by $[\xi]$ the corresponding element in ${\rm Ext}^1_\mathcal{A}(Y, K)$.

Let $\xi\colon 0\rightarrow K \rightarrow X\rightarrow Y\rightarrow 0$ be an exact sequence and $Z$ be an object. We consider the \emph{connecting map}
$$c(\xi, Z)\colon {\rm Hom}_\mathcal{A}(K, Z)\longrightarrow {\rm Ext}_\mathcal{A}^1(Y, Z)$$
 sending $u\colon K\rightarrow Z$ to $[u.\xi]$, where $u.\xi$ denotes the pushout of $\xi$ along $u$. The connecting map $c(\xi, Z)$ is a morphism of $\Gamma(Z)$-modules, and thus its image ${\rm Im}\; c(\xi, Z)$ is a $\Gamma(Z)$-submodule of ${\rm Ext}_\mathcal{A}^1(Y, Z)$.

For an object $K$, denote by ${\rm add}\; K$ the full subcategory of $\mathcal{A}$ formed by direct summands of finite direct sums of $K$. Recall
that
$${\rm Hom}_\mathcal{A}(-, K)\colon {\rm add} \;K\longrightarrow \Gamma(K)\mbox{-proj}$$
is a duality of categories, which sends $K$ to the regular module $\Gamma(K)$. Here, for any ring $R$, $R\mbox{-proj}$ denotes the category of finitely generated projective left $R$-modules.

The following two facts are standard.

\begin{lem}\label{lem:ext}
Let $K, K'$ and $Y$ be objects in $\mathcal{A}$ with $K'\in {\rm add}\; K$. Then the following statements hold:
\begin{enumerate}
\item there is a natural isomorphism
\begin{align*}
{\rm Hom}_\mathcal{A}(Y, K')\stackrel{\sim}\longrightarrow {\rm Hom}_{\Gamma(K)} ({\rm Hom}_\mathcal{A}(K', K), {\rm Hom}_\mathcal{A}(Y, K))
\end{align*}
sending $u\colon Y\rightarrow K'$ to  ${\rm Hom}_\mathcal{A}(u, K)$;
\item there is a natural isomorphism
\begin{align*}
{\rm Ext}_\mathcal{A}^1(Y, K')\stackrel{\sim}\longrightarrow {\rm Hom}_{\Gamma(K)} ({\rm Hom}_\mathcal{A}(K', K), {\rm Ext}_\mathcal{A}^1(Y, K))
\end{align*}
sending $[\xi]\in {\rm Ext}^1_\mathcal{A}(Y, K')$ to the connecting map $c(\xi, K)$.
\end{enumerate}\end{lem}

\begin{proof}
Both isomorphisms are special cases of the following well-known fact: for an additive functor $F$ from ${\rm add}\; K$ to the category of abelian groups, there is an isomorphism
\begin{align*}
F(K')\stackrel{\sim}\longrightarrow {\rm Hom}_{\Gamma(K)} ({\rm Hom}_\mathcal{A}(K', K), F(K))
\end{align*}
sending $a\in F(K')$ to a $\Gamma(K)$-module morphism that sends $f\colon K'\rightarrow K$ to $F(f)(a)\in F(K)$.
\end{proof}

\begin{lem}\label{lem:factor}
Let $K, Y$ be two objects in the abelian category $\mathcal{A}$. Consider two exact sequences $\xi_1\colon 0\rightarrow K_1\rightarrow X_1\stackrel{\alpha_1}\rightarrow Y\rightarrow 0$ and
$\xi_2\colon 0\rightarrow K_2\rightarrow X_2\stackrel{\alpha_2}\rightarrow Y\rightarrow 0$, and  the following statements:
\begin{enumerate}
\item there is a morphism $v\colon X_1\rightarrow X_2$ such that $\alpha_1=\alpha_2\circ v$;
\item there is a morphism $u\colon K_1\rightarrow K_2$ such that $[\xi_2]=[u.\xi_1]$;
\item ${\rm Im}\; c(\xi_2, K)\subseteq {\rm Im}\; c(\xi_1, K)$.
\end{enumerate}
Then we have $(1)\Leftrightarrow (2)\Rightarrow(3)$. Consequently, if $\alpha_1$ and $\alpha_2$ are right equivalent, then ${\rm Im}\; c(\xi_1, K)= {\rm Im}\; c(\xi_2, K)$.

Moreover,  $K_2\in {\rm add}\; K$ implies $(3)\Rightarrow (2)$. Consequently, if both $K_1$ and $K_2$ lie in ${\rm add}\; K$, then we have that $\alpha_1$ and $\alpha_2$ are right equivalent if and only if  ${\rm Im}\; c(\xi_1, K)= {\rm Im}\; c(\xi_2, K)$.
\end{lem}

\begin{proof}
The implications ``$(1)\Leftrightarrow (2)$" and ``$(2)\Rightarrow (3)$" are direct.

For ``$(3)\Rightarrow (2)$" in the case that $K_2\in {\rm add}\; K$, we consider the two epimorphisms ${\rm Hom}_\mathcal{A}(K_1, K)\rightarrow {\rm Im}\; c(\xi_1, K)$ and ${\rm Hom}_\mathcal{A}(K_2, K)\rightarrow {\rm Im}\; c(\xi_2, K)$ of $\Gamma(K)$-modules.  We apply the projectivity of the $\Gamma(K)$-module ${\rm Hom}_\mathcal{A}(K_2, K)$ and the condition ${\rm Im}\; c(\xi_2, K)\subseteq {\rm Im}\; c(\xi_1, K)$. Then  there exists a morphism $w\colon {\rm Hom}_\mathcal{A}(K_2, K)\rightarrow {\rm Hom}_\mathcal{A}(K_1, K)$ such that $c(\xi_2, K)=c(\xi_1, K)\circ w$. By Lemma \ref{lem:ext}(1), there exists a unique morphism $u\colon K_1\rightarrow K_2$ with $w={\rm Hom}_\mathcal{A}(u, K)$. Then we obtain  $c(\xi_2, K)=c(\xi_1, K)\circ {\rm Hom}_\mathcal{A}(u, K)$. It implies that for any morphism $x\colon K_2\rightarrow K$ we have $[x.\xi_2]=[x.(u.\xi_1)]$.

We claim that $[\xi_2]=[u.\xi_1]$. Indeed, since $K_2\in {\rm add}\; K$, there exist morphisms $x_i\colon K_2\rightarrow K$ and $y_i\colon K\rightarrow K_2$, $1\leq i\leq n$, such that ${\rm Id}_{K_2}=\sum_{i=1}^ny_i\circ x_i$. Then we have
\begin{align*}
[\xi_2]=[{\rm Id}_{K_2}.\xi]=\sum_{i=1}^n[y_i.(x_i.\xi_2)]=\sum_{i=1}^n[y_i.(x_i.(u.\xi_1))]=[u.\xi_1].
\end{align*}
The remaining statements are immediate consequences.
\end{proof}

The following existence result is the basis of our treatment.

\begin{prop}\label{prop:semi-uni}
Let $K, Y$ be objects in $\mathcal{A}$. Then for any finitely generated $\Gamma(K)$-submodule $L\subseteq {\rm Ext}_\mathcal{A}^1(Y, K)$, there exists an exact sequence $\xi\colon 0\rightarrow K_1\rightarrow X\stackrel{\alpha}\rightarrow Y\rightarrow 0$ with $K_1\in {\rm add}\; K$ and ${\rm Im}\; c(\xi, K)=L$.
\end{prop}

\begin{proof}
Recall the duality of categories  ${\rm Hom}_\mathcal{A}(-, K)\colon {\rm add} \;K\rightarrow \Gamma(K)\mbox{-proj}$. Since the $\Gamma(K)$-module $L$ is finitely generated,  we may take a morphism $$c\colon {\rm Hom}_\mathcal{A}(K_1, K)\longrightarrow {\rm Ext}_\mathcal{A}^1(Y, K)$$ with image $L$ and $K_1\in {\rm add}\; K$. By Lemma \ref{lem:ext}(2), there exists an exact sequence $\xi\colon 0\rightarrow K_1\rightarrow X\rightarrow Y\rightarrow 0$ such that $c=c(\xi, K)$.
\end{proof}

We call the above exact sequence $\xi$ a \emph{semi-universal $L$-extension} of $Y$ by $K$. In case that ${\rm Ext}_\mathcal{A}^1(Y, K)$ is a finitely generated $\Gamma(K)$-module, a semi-universal ${\rm Ext}_\mathcal{A}^1(Y, K)$-extension is called a \emph{semi-universal extension} of $Y$ by $K$.

  Semi-universal $L$-extensions enjoy a certain (semi-)universal property: if $\xi'\colon 0\rightarrow K'_1\rightarrow X'\stackrel{\alpha'}\rightarrow Y\rightarrow 0$ is  an exact sequence with $K'_1\in {\rm add}\; K$, then by Lemma \ref{lem:factor}   $\alpha$ factors through $\alpha'$ if and only if ${\rm Im}\; c(\xi', K)\subseteq L$; if $\xi'$ is also a semi-universal $L$-extension, then the two morphisms $\alpha$ and $\alpha'$ are right equivalent, or equivalently, there exist morphisms $u\colon K_1\rightarrow K_1'$ and $v\colon K_1'\rightarrow K_1$ satisfying $[\xi']=[u.\xi]$ and $[\xi]=[v.\xi']$.

Let $K, Y$ be objects in $\mathcal{A}$. Recall from the introduction the poset $[\longrightarrow Y\rangle_{\rm epi}$ consisting of right equivalence classes $[\alpha\rangle$ of epimorphisms $\alpha\colon X\rightarrow Y$.  For a left module $M$ over a ring $R$, denote by ${\rm Sub}_{R}\; M$ the poset of $R$-submodules of $M$, and by ${\rm sub}_R\; M$ the subset consisting of finitely generated $R$-submodules.

For $[\alpha\rangle \in [\longrightarrow Y\rangle_{\rm epi}$, we consider the corresponding exact sequence $\xi_\alpha\colon 0\rightarrow {\rm Ker}\; \alpha\rightarrow X\stackrel{\alpha}\rightarrow Y\rightarrow 0$. Define $\delta_{K, Y}([\alpha\rangle)={\rm Im}\; c(\xi_\alpha, K)$; it is a $\Gamma(K)$-submodule of ${\rm Ext}_\mathcal{A}^1(Y, K)$.  By Lemma \ref{lem:factor}, $\delta_{K, Y}([\alpha\rangle)$ is independent of the choice of the representative $\alpha$ of the class $[\alpha\rangle$. This gives rise to the following well-defined map
\begin{align}
\delta_{K, Y}\colon {[\longrightarrow Y\rangle_{\rm epi}} \longrightarrow {\rm Sub}_{\Gamma(K)}\; {\rm Ext}^1_\mathcal{A}(Y, K).
\end{align}

We denote by ${_K[\longrightarrow Y\rangle_{\rm epi}}$ the subset of $[\longrightarrow Y\rangle_{\rm epi}$ formed by those classes $[\alpha\rangle$ that have a representative $\alpha\colon X\rightarrow Y$ with the kernel in ${\rm add}\; K$. Then $\delta_{K, Y}([\alpha\rangle)$ is a finitely generated $\Gamma(K)$-module, since ${\rm Hom}_\mathcal{A}({\rm Ker}\; \alpha, K)$ is a finitely generated projective $\Gamma(K)$-module.

\begin{prop}\label{prop:delta}
Let $\mathcal{A}$ be an abelian category. Keep the notation as above. Then there is a bijection
\begin{align}
\delta_{K, Y}\colon {_K[\longrightarrow Y\rangle_{\rm epi}} \longrightarrow {\rm sub}_{\Gamma(K)}\; {\rm Ext}^1_\mathcal{A}(Y, K).
\end{align}
Moreover, $\delta_{K, Y}$ is an anti-isomorphism of posets.
\end{prop}

\begin{proof}
The injectivity of $\delta_{K, Y}$ follows from Lemma \ref{lem:factor}. For the surjectivity, let $L$ be a finitely generated $\Gamma(K)$-submodule of ${\rm Ext}^1_\mathcal{A}(Y, K)$. We consider a semi-universal $L$-extension $\xi\colon 0\rightarrow K_1\rightarrow X\stackrel{\alpha}\rightarrow Y\rightarrow 0$ as in Proposition \ref{prop:semi-uni}; in particular, $[\alpha\rangle$ lies in $_K[\longrightarrow Y\rangle_{\rm epi}$. Then $\delta_{K, Y}([\alpha\rangle)=L$. The bijection $\delta_{K, Y}$ reverses the orders of the two posets by Lemma \ref{lem:factor}.
\end{proof}

\section{Right minimal epimorphisms and universal extensions}

In this section, we study right minimal epimorphisms and define universal extensions. A semi-universal extension is a universal extension if and only if the corresponding epimorphism is right minimal.

Throughout $\mathcal{A}$ is an abelian category. Recall from  \cite[I.2]{ARS} that a morphism $\alpha\colon X\rightarrow Y$ is \emph{right minimal} provided that any endomorphism $u\colon X\rightarrow X$ satisfying $\alpha=\alpha\circ u$ is an automorphism. We observe that an epimorphism $P\rightarrow M$ of $R$-modules with $P$ projective is a projective cover if and only if it is right minimal.

We need the following well-known observation.

\begin{lem}\label{lem:rm}
Let $\xi\colon 0\rightarrow K\rightarrow X\stackrel{\alpha}\rightarrow Y\rightarrow 0$ be an exact sequence in $\mathcal{A}$. Then the following statements are equivalent:
\begin{enumerate}
\item the epimorphism  $\alpha$ is right minimal;
 \item any endomorphism $u\colon K\rightarrow K$ satisfying $[\xi]=[u.\xi]$ is an automorphism;
 \item the connecting map $c(\xi, K)\colon {\rm Hom}_\mathcal{A}(K, K)\rightarrow {\rm Ext}^1_\mathcal{A}(Y, K)$ of $\Gamma(K)$-modules is right minimal.
 \end{enumerate}
\end{lem}

\begin{proof}
For the equivalence ``$(2)\Leftrightarrow(3)$", we observe that any morphism $\Gamma(K)={\rm Hom}_\mathcal{A}(K, K) \rightarrow M$ of $\Gamma(K)$-modules is uniquely determined by the image of ${\rm Id}_K$, and thus any endomorphism ${\rm Hom}_\mathcal{A}(K, K)\rightarrow {\rm Hom}_\mathcal{A}(K, K)$ is of the form ${\rm Hom}_A(u, K)$ for a unique morphism $u\colon K\rightarrow K$. It follows that $c(\xi, K)=c(\xi, K)\circ {\rm Hom}_\mathcal{A}(u, K)$ if and only if $[\xi]=[u.\xi]$. We infer from this the required equivalence.

For ``$(1)\Rightarrow (2)$", we observe that the identity $[\xi]=[u.\xi]$ implies the existence of the following commutative diagram
\[\xymatrix{
0 \ar[r] & K\ar[d]^u \ar[r] & X\ar[r]^\alpha \ar@{.>}[d]^v & Y \ar[r] \ar@{=}[d] & 0 \\
0 \ar[r] & K \ar[r] & X\ar[r]^\alpha & Y \ar[r] & 0.
}\]
Since $\alpha=\alpha \circ v$ and $\alpha$ is right minimal, then $v$ is an automorphism. It follows that $u$ is an automorphism. For ``$(2)\Rightarrow (1)$", we just reverse the argument.
\end{proof}

The following characterization of right minimal epimorphisms extends slightly the above equivalence ``$(1)\Leftrightarrow (3)$".

\begin{prop}\label{prop:rm}
Let $\xi\colon 0\rightarrow K\rightarrow X\stackrel{\alpha}\rightarrow Y\rightarrow 0$ be an exact sequence, and let $K'$ be an object with  $K\in {\rm add}\; K'$. Then $\alpha$ is right minimal if and only if the connecting map $c(\xi, K')\colon {\rm Hom}_\mathcal{A}(K, K')\rightarrow {\rm Ext}^1_\mathcal{A}(Y, K')$ of $\Gamma(K')$-modules is right minimal.
\end{prop}

\begin{proof}
For the ``only if" part, recall the duality ${\rm Hom}_\mathcal{A}(-, K')\colon {\rm add} \;K'\rightarrow \Gamma(K')\mbox{-proj}$. Then any morphism ${\rm Hom}_\mathcal{A}(K, K')\rightarrow {\rm Hom}_\mathcal{A}(K, K')$ of $\Gamma(K')$-modules is of the form ${\rm Hom}_\mathcal{A}(u, K')$ for $u\colon K\rightarrow K$. We assume that $c(\xi, K')=c(\xi, K')\circ {\rm Hom}_\mathcal{A}(u, K')$, that is, any morphism $x\colon K\rightarrow K'$ satisfies $[x.\xi]=[x.(u.\xi)]$. We will show that $u$, thus ${\rm Hom}_\mathcal{A}(u, K')$, is an automorphism.

Since $K\in {\rm add}\; K'$, there exist morphisms $x_i\colon K\rightarrow K'$ and $y_i\colon K'\rightarrow K$, $1\leq i\leq n$, satisfying ${\rm Id}_K=\sum_{i=1}^n y_i\circ x_i$. Then we have
\begin{align*}
[\xi]=[{\rm Id}_K. \xi]=\sum_{i=1}^n[y_i.(x_i.\xi)]=\sum_{i=1}^n [y_i.(x_i.(u. \xi))]=[u.\xi].
\end{align*}
Since $\alpha$ is right minimal, we infer that $u$ is an automorphism by Lemma \ref{lem:rm}(2).

For the ``if" part, it suffices to show that any endomorphism $u\colon K\rightarrow K$ with $[\xi]=[u.\xi]$ is an automorphism; see Lemma \ref{lem:rm}. Then we have $c(\xi, K')=c(\xi, K')\circ {\rm Hom}_A(u, K')$. By the right minimality of $c(\xi, K')$, we infer that ${\rm Hom}_A(u, K')$, thus $u$, is an automorphism.
\end{proof}

The following result seems to be well known.

\begin{cor}\label{cor:rerm}
Let $\alpha\colon X\rightarrow Y$ be a morphism  and $K$ an object with ${\rm Ker}\; \alpha\in {\rm add}\; K$. Consider the corresponding exact sequence $\xi_\alpha\colon 0\rightarrow {\rm Ker}\; \alpha\rightarrow X\rightarrow {\rm Im}\; \alpha\rightarrow 0$. Then $\alpha$ is right equivalent to a right minimal morphism if and only if
the $\Gamma(K)$-module ${\rm Im}\; c(\xi_\alpha, K)$ has a projective cover.
\end{cor}

\begin{proof}
Without loss of generality, we assume that $\alpha$ is epic. Indeed, $\alpha\colon X\rightarrow Y$ is right equivalent to $\alpha'\colon X'\rightarrow Y$ if and only if ${\rm Im}\; \alpha={\rm Im}\; \alpha'$ and the induced epimorphisms $X\rightarrow {\rm Im}\; \alpha$ and $X'\rightarrow {\rm Im}\; \alpha'$ are right equivalent; moreover, $\alpha'$ is right minimal if and only if so is $X'\rightarrow {\rm Im}\; \alpha'$.

For the ``only if" part, we assume that $\alpha$ is right equivalent to a right minimal morphism $\alpha'\colon X'\rightarrow Y$. Observe that $\alpha'$ is epic and we have the corresponding exact sequence $\xi_{\alpha'}\colon 0\rightarrow {\rm Ker}\; \alpha'\rightarrow X'\rightarrow Y\rightarrow 0$. By Lemma \ref{lem:factor}(1) and (2), there exist morphisms $u\colon {\rm Ker}\; \alpha\rightarrow {\rm Ker}\; \alpha'$ and $v\colon {\rm Ker}\; \alpha'\rightarrow {\rm Ker}\; \alpha$ such that $[\xi_{\alpha'}]=[u.\xi_\alpha]$ and $[\xi_\alpha]=[v. \xi_{\alpha'}]$. By $[\xi_{\alpha'}]=[(u\circ v).\xi_{\alpha'}]$ and the right minimality of $\alpha'$, we apply Lemma \ref{lem:rm} to obtain that $u\circ v$ is an isomorphism. In particular, ${\rm Ker}\; \alpha'$ is a direct summand of ${\rm Ker}\; \alpha$, hence belongs to ${\rm add}\; K$. By Lemma \ref{lem:factor}, ${\rm Im}\; c(\xi_\alpha, K)={\rm Im}\; c(\xi_{\alpha'}, K)$. By Proposition \ref{prop:rm} the morphism ${\rm Hom}_\mathcal{A}({\rm Ker}\; \alpha', K)\rightarrow {\rm Im}\; c(\xi_{\alpha'}, K)$ of $\Gamma(K)$-modules is a projective cover.

For the ``if" part, we take a projective cover ${\rm Hom}_\mathcal{A}(K', K) \rightarrow {\rm Im}\; c(\xi_\alpha, K)$ of $\Gamma(K)$-modules with $K'\in {\rm add}\; K$. Thus we obtain a right minimal morphism $c\colon {\rm Hom}_\mathcal{A}(K', K)\rightarrow {\rm Ext}_\mathcal{A}^1(Y, K)$ of $\Gamma(K)$-modules, which corresponds by Lemma \ref{lem:ext}(2) to an exact sequence $\xi'\colon 0\rightarrow K'\rightarrow X'\stackrel{\alpha'}\rightarrow Y\rightarrow 0$; in particular, $c=c(\xi', K)$ and ${\rm Im}\; c(\xi_\alpha, K)={\rm Im}\; c(\xi', K)$. By Proposition \ref{prop:rm}, the morphism $\alpha'$ is right minimal, and by Lemma \ref{lem:factor} $\alpha$ is right equivalent to $\alpha'$.
\end{proof}

The above proposition suggests the following definition. The notion does not come as a surprise, since universal extensions are widely used in the representation theory of artin algebras to construct modules with prescribed homological properties.

\begin{defn}\label{defn:uni-ext}
Let $K, Y$ be two objects in $\mathcal{A}$ and $L\subseteq {\rm Ext}_\mathcal{A}^1(Y, K)$ be a finitely generated $\Gamma(K)$-submodule. A \emph{universal $L$-extension} of $Y$ by $K$ is an exact sequence $\xi\colon 0\rightarrow K'\rightarrow X \rightarrow Y\rightarrow 0$ satisfying the following conditions:
\begin{enumerate}
\item[(UE1)] $K'\in {\rm add}\; K$;
 \item[(UE2)] ${\rm Im}\; (c(\xi, K)\colon {\rm Hom}_A(K', K)\rightarrow {\rm Ext}_\mathcal{A}^1(Y, K))=L$;
 \item[(UE3)] the connecting map $c(\xi, K)\colon {\rm Hom}_A(K', K)\rightarrow {\rm Ext}_\mathcal{A}^1(Y, K)$ of $\Gamma(K)$-modules is right minimal.
\end{enumerate}
In  case that ${\rm Ext}_\mathcal{A}^1(Y, K)$ is a finitely generated $\Gamma(K)$-module, a universal ${\rm Ext}_\mathcal{A}^1(Y, K)$-extension is called a \emph{universal extension} of $Y$ by $K$. \hfill $\square$
\end{defn}

We summarize some properties of universal extensions in the following proposition. In particular, the second statement implies that universal extensions are unique up to certain isomorphisms.

\begin{prop}\label{prop:uni-ext}
Let $K, Y$ be objects in $\mathcal{A}$ and let $L\subseteq {\rm Ext}_\mathcal{A}^1(Y, K)$ be a finitely generated $\Gamma(K)$-submodule. Let $\xi_1\colon 0\rightarrow K_1\rightarrow X_1\stackrel{\alpha_1}\rightarrow Y\rightarrow 0$ and $\xi_2\colon 0\rightarrow K_2\rightarrow X_2\stackrel{\alpha_2}\rightarrow Y\rightarrow 0$ be exact sequences in $\mathcal{A}$. Then the following statements hold.
\begin{enumerate}
\item The exact sequence $\xi_1$ is a universal $L$-extension if and only if it is a semi-universal $L$-extension and $\alpha_1$ is right minimal.
\item If $\xi_1$ is a universal $L$-extension, then $\xi_2$ is a universal $L$-extension if and only if there exists an isomorphism $u\colon K_1\rightarrow K_2$ such that $[\xi_2]=[u.\xi_1]$.
\item There exists a universal $L$-extension if and only if the $\Gamma(K)$-module $L$ has a projective cover.
\end{enumerate}
\end{prop}

\begin{proof}
(1) follows from Proposition \ref{prop:rm}. For (2), it suffices to prove the ``only if" part. In this case, both $\xi_1$ and $\xi_2$ are semi-universal $L$-extensions. By Lemma \ref{lem:factor} there exist $u\colon K_1\rightarrow K_2$ and $u'\colon K_2\rightarrow K_1$ satisfying $[\xi_2]=[u.\xi_1]$ and $[\xi_1]=[u'.\xi_2]$. Then we have $[\xi_1]=[(u'\circ u).\xi_1]$ and $[\xi_2]=[(u\circ u').\xi_2]$. By Lemma \ref{lem:rm}, both $u'\circ u$ and $u\circ u'$ are automorphisms. We infer that $u$ is an isomorphism.

By Proposition \ref{prop:rm} it suffices to prove the ``if" part of (3). We take a morphism $c\colon {\rm Hom}_\mathcal{A}(K', K)\rightarrow {\rm Ext}_\mathcal{A}^1(Y, K)$ of $\Gamma(K)$-modules with $K'\in {\rm add}\; K$ and ${\rm Im}\; c=L$; moreover, the induced epimorphism ${\rm Hom}_\mathcal{A}(K', K)\rightarrow L$ is a projective cover. In particular, the morphism $c$ is right minimal. By Lemma \ref{lem:ext}(2), there exists an exact sequence $\xi\colon 0\rightarrow K'\rightarrow X\rightarrow Y\rightarrow 0$ with the property $c=c(\xi, K)$. Then $\xi$ is a universal $L$-extension.
\end{proof}

\section{The Auslander-Reiten duality and a bijection triangle}

In this section, we exploit a bijection triangle in an abelian category having Auslander-Reiten duality; see Theorem \ref{thm:1}. It implies that the restricted Auslander bijection holds in an abelian category having Auslander-Reiten duality.

Let $k$ be a commutative artinian ring. We denote by $D={\rm Hom}_k(-, E)$ the duality on finitely generated $k$-modules, where $E$ is the minimal injective cogenerator of $k$.

Let $\mathcal{A}$ be a Hom-finite $k$-linear abelian category. The Hom-finiteness means that the $k$-module ${\rm Hom}_\mathcal{A}(X, Y)$ is finitely generated for each pair $X, Y$ of objects. In particular, for any object $K$ the endomorphism ring $\Gamma(K)$ is an artin algebra and thus any $\Gamma(K)$-module has a projective cover. As a consequence of  Corollary \ref{cor:rerm}, any morphism in $\mathcal{A}$ is right equivalent to a right minimal morphism.

Following \cite{LZ}, a morphism $\alpha\colon X\rightarrow Y$ is \emph{projectively trivial} if it factors through any epimorphism ending at $Y$. This is equivalent to ${\rm Ext}_\mathcal{A}^1(\alpha, -)=0$. The latter means that $[\xi.\alpha]=0$ for  any extension $[\xi]\in {\rm Ext}_\mathcal{A}^1(Y, Z)$ and any object $Z$. Here, $\xi.\alpha$ denotes the pullback of $\xi$ along $\alpha$.

For any objects $X$ and $Y$, we denote by $\mathcal{P}(X, Y)$ the $k$-submodule  of ${\rm Hom}_\mathcal{A}(X, Y)$ formed by projectively trivial morphisms. This gives rise to an ideal $\mathcal{P}$ of $\mathcal{A}$ and the corresponding factor category is denoted by $\underline{\mathcal{A}}$; the Hom spaces in $\underline{\mathcal{A}}$ are denoted by $\underline{\rm Hom}_\mathcal{A}(X, Y)$. Hence, we have $\underline{\rm Hom}_\mathcal{A}(X, Y)={\rm Hom}_\mathcal{A}(X, Y)/\mathcal{P}(X, Y)$.

Dually, one defines \emph{injectively trivial} morphisms, the ideal $\mathcal{I}$ of injectively trivial morphisms, the factor category $\overline{\mathcal{A}}$ and the notation $\overline{\rm Hom}_\mathcal{A}(X, Y)$; compare \cite[IV.1]{ARS}.

The abelian category $\mathcal{A}$ is said to \emph{have Auslander-Reiten duality} \cite{LZ} provided that there exists a $k$-linear equivalence $\tau\colon \underline{\mathcal{A}}\rightarrow \overline{\mathcal{A}}$ with a $k$-linear functorial isomorphism
\begin{align}\label{equ:AR}
\Phi_{X, Y}\colon D{\rm Ext}_\mathcal{A}^1(X, Y) \stackrel{\sim}\longrightarrow \overline{\rm Hom}_\mathcal{A}(Y, \tau X).
\end{align}
 The equivalence $\tau$ is called the \emph{Auslander-Reiten translation} of $\mathcal{A}$; it is unique up to a natural isomorphism.

\begin{exm}
{\rm Let $\mathcal{A}$ be a  Hom-finite $k$-linear abelian category. By \cite[Theorem 1.1]{LZ}, $\mathcal{A}$ has Auslander-Reiten duality if and only if it has almost split sequences. It follows from \cite[Proposition 3.2]{AR} that the category of finitely presented functors on a dualizing $k$-variety has Auslander-Reiten duality. In particular, the category of finitely generated modules over an artin algebra has Auslander-Reiten duality.

For another class of examples, let $Q$ be a connected locally finite interval-finite quiver and let $k$ be a field. Then the category ${\rm rep}(Q)$ of finitely presented representations of $Q$  over $k$ has Auslander-Reiten duality if and only if either $Q$ has neither left infinite paths nor right infinite paths, or $Q$ itself is a left infinite path. For details, we refer to \cite[Theorem 3.7]{BLP} and \cite[Corollary 4.5]{Ji}.

 We observe that $\mathcal{A}$ has Serre duality in the sense of \cite[Theorem 5.2]{LZ} if and only if it has Auslander-Reiten duality and any projectively trivial or injectively trivial morphism is zero; compare \cite[Lemma 5.1]{LZ}. \hfill $\square$}
\end{exm}

Let $\mathcal{A}$ have the Auslander-Reiten duality (\ref{equ:AR}).  We infer that the $k$-module ${\rm Ext}_\mathcal{A}^1(X, Y)$ is finitely generated for each pair $X, Y$ of objects, and thus any $\Gamma(Y)$-submodule $L$ of ${\rm Ext}_\mathcal{A}^1(X, Y)$ is finitely generated. Hence, universal $L$-extensions in $\mathcal{A}$ are always defined and they exist by Proposition \ref{prop:uni-ext}(3).

Denote by $\tau^{-1}$ a quasi-inverse of $\tau$. Moreover, we fix the counit $\varepsilon\colon \tau^{-1}\tau\rightarrow {\rm Id}_{\underline{\mathcal{A}}}$ and the unit $\eta\colon {\rm Id}_{\overline{\mathcal{A}}}\rightarrow \tau \tau^{-1}$ for the adjoint pair $(\tau^{-1}, \tau)$.  The isomorphism (\ref{equ:AR}) induces another functorial isomorphism
\begin{align}\label{equ:AR2}
\Psi_{X, Y}\colon D{\rm Ext}_\mathcal{A}^1(X, Y) \stackrel{\Phi_{X, Y}}\longrightarrow \overline{\rm Hom}_\mathcal{A}(Y, \tau X) \longrightarrow \underline{\rm Hom}_\mathcal{A}(\tau^{-1}Y, X).
\end{align}
Here, the right isomorphism is given by the inverse of the adjunction, which sends $f\colon Y\rightarrow \tau X$ to $\varepsilon_X\circ \tau^{-1}(f)$. Similarly, we have a functorial isomorphism
\begin{align}\label{equ:AR3}
\Phi'_{X, Y}\colon D{\rm Ext}_\mathcal{A}^1(\tau^{-1}X, Y) \xrightarrow{\Phi_{\tau^{-1}X, Y}} \overline{\rm Hom}_\mathcal{A}(Y, \tau \tau^{-1}X) \longrightarrow \overline{\rm Hom}_\mathcal{A}(Y, X),
\end{align}
where the right isomorphism is given by $\overline{\rm Hom}_\mathcal{A}(Y, (\eta_X)^{-1})$.

We recall that for a module $M$ over a ring $R$, we denote by ${\rm sub}_{R}\; M$ the poset of finitely generated $R$-submodules of $M$. We observe the following immediate consequence of the Auslander-Reiten duality (\ref{equ:AR2}).

\begin{lem}\label{lem:AR-bij}
Let $\mathcal{A}$ have Auslander-Reiten duality, and let $K, Y$ be objects in $\mathcal{A}$. Then there is a bijection
\begin{align}\label{equ:bij-AR}
\gamma_{K, Y}\colon {\rm sub}_{\Gamma(K)} \; {\rm Ext}_\mathcal{A}^1(Y, K)\longrightarrow {\rm sub}_{\Gamma(K)^{\rm op}} \; \underline{\rm Hom}_\mathcal{A}(\tau^{-1} K, Y)
\end{align}
such that for any $\Gamma(K)$-submodule $L$ of  ${\rm Ext}_\mathcal{A}^1(Y, K)$, $\gamma_{K, Y}(L)=H$ is defined by the following exact sequence
\begin{align}\label{equ:exact}
0\longrightarrow H\stackrel{\rm inc}\longrightarrow \underline{\rm Hom}_\mathcal{A}(\tau^{-1} K, Y)\stackrel{\psi}\longrightarrow DL\longrightarrow 0,
\end{align}
where ``${\rm inc}$" denotes the inclusion and $\psi=D({\rm inc}')\circ (\Psi_{Y, K})^{-1}$. Here, ${\rm inc}'\colon L\rightarrow  {\rm Ext}_\mathcal{A}^1(Y, K)$ is the inclusion. The bijection is an anti-isomorphism of posets.
\end{lem}

\begin{proof}
We use the fact that for a finitely generated module $M$ over an artin algebra $A$, there is a bijection between ${\rm sub}_{A}\; M$ and ${\rm sub}_{A^{\rm op}}\; DM$, sending a submodule $L$ to the kernel of the projection $DM\rightarrow DL$.  Then we identify $D{\rm Ext}_\mathcal{A}^1(Y, K)$ with $\underline{\rm Hom}_\mathcal{A}(\tau^{-1}K, Y)$ via $\Psi_{Y, K}$.
\end{proof}

Let $\xi\colon 0\rightarrow K'\rightarrow X\rightarrow Y\rightarrow 0$ be an exact sequence and let $K$ be an object. We consider the connecting map $c(\xi, K)\colon {\rm Hom}_\mathcal{A}(K', K)\rightarrow {\rm Ext}_\mathcal{A}^1(Y, K)$, which vanishes on $\mathcal{I}(K', K)$. Hence it induces
$$\overline{\rm Hom}_\mathcal{A}(K', K)\longrightarrow {\rm Ext}_\mathcal{A}^1(Y, K),$$
which is still denoted by $c(\xi, K)$. Dually, we have the \emph{connecting map}
 $$c(K, \xi)\colon {\rm Hom}_\mathcal{A}(K, Y)\longrightarrow {\rm Ext}_\mathcal{A}^1(K, K')$$ sending $u\colon K\rightarrow Y$ to $[\xi.u]$. Here, $\xi.u$ denotes the pullback of $\xi$ along $u$. It vanishes on $\mathcal{P}(K, Y)$. Then we have the induced map $$c(K, \xi)\colon \underline{\rm Hom}_\mathcal{A}(K, Y)\longrightarrow {\rm Ext}_\mathcal{A}^1(K, K').$$

We observe the following compatibility property in the Auslander-Reiten duality  (\ref{equ:AR2}) and (\ref{equ:AR3}), which is implicitly contained in the argument of \cite[Section 3]{LZ}; compare \cite[Section 3]{Ji}.

\begin{lem}\label{lem:compat}
Let $\mathcal{A}$ have Auslander-Reiten duality. Then for any exact sequence  $\xi\colon 0\rightarrow K'\rightarrow X\rightarrow Y\rightarrow 0$ and any object $K$, there is a commutative diagram
\[\xymatrix{
D{\rm Ext}_\mathcal{A}^1(Y, K) \ar[d]_{\Psi_{Y, K}} \ar[rr]^{Dc(\xi, K)} && D\overline{\rm Hom}_\mathcal{A}(K', K) \ar[d]^{D\Phi'_{K, K'}}\\
\underline{\rm Hom}_\mathcal{A}(\tau^{-1} K, Y) \ar[rr]^{c(\tau^{-1}K, \xi)} && {\rm Ext}^1_\mathcal{A}(\tau^{-1}K, K').
}\]
In particular, we have an exact sequence
\begin{align}\label{equ:compat}
0\rightarrow {\rm Ker}\; c(\tau^{-1}K, \xi) \stackrel{\rm inc}\longrightarrow \underline{\rm Hom}(\tau^{-1}K, Y) \stackrel{\psi}\longrightarrow D{\rm Im}\; c(\xi, K)\rightarrow 0,
\end{align}
where ``{\rm inc}" denotes the inclusion and $\psi=D({\rm inc}')\circ (\Psi_{Y, K})^{-1}$. Here, ${\rm inc}'\colon {\rm Im}\; c(\xi, K)\rightarrow {\rm Ext}_\mathcal{A}^1(Y, K) $ is the inclusion.
\end{lem}

\begin{proof}
The compatibility follows from the explicit construction of the Auslander-Reiten duality in \cite[Proposition 3.1 and Lemma 3.2]{LZ}. In particular, the identity in \cite[Lemma 3.2]{LZ} is used.

By the commutative diagram, the kernel of $c(\tau^{-1}K, \xi)$ is identified via $\Psi_{Y, K}$ to the kernel of $D(c(\xi, K))$, which equals  the kernel of $D({\rm inc}')$. This yields the exact sequence  (\ref{equ:compat}).
\end{proof}

For any objects $C$ and $Y$ the following map is well defined
\begin{align*}
\eta_{C, Y}\colon [\longrightarrow Y\rangle_{\rm epi} \longrightarrow {\rm sub}_{\Gamma(C)^{\rm op}} \; \underline{\rm Hom}_\mathcal{A}(C, Y),
\end{align*}
which sends $[\alpha\rangle$ to ${\rm Im}\; \underline{\rm Hom}_\mathcal{A}(C, \alpha)$ for any epimorphism $\alpha\colon X\rightarrow Y$. We observe that ${\rm Im}\; \underline{\rm Hom}_\mathcal{A}(C, \alpha)={\rm Im }\; {\rm Hom}_\mathcal{A}(C, \alpha)/\mathcal{P}(C, Y)$.

For another object $K$, we identify via $\tau^{-1}$ the  $\Gamma(\tau^{-1}K)^{\rm op}$-module structure on $\underline{\rm Hom}_\mathcal{A}(\tau^{-1}K, Y)$ with the corresponding $\Gamma(K)^{\rm op}$-module structure. Hence, by abuse of notation, we identify the poset ${\rm sub}_{\Gamma(\tau^{-1}K)^{\rm op}} \; \underline{\rm Hom}_\mathcal{A}(\tau^{-1}K, Y)$ with ${\rm sub}_{\Gamma(K)^{\rm op}} \; \underline{\rm Hom}_\mathcal{A}(\tau^{-1}K, Y)$. By this identification, we have the bijection (\ref{equ:bij-AR})
$$\gamma_{K, Y}\colon {\rm sub}_{\Gamma(K)} \; {\rm Ext}_\mathcal{A}^1(Y, K)\longrightarrow {\rm sub}_{\Gamma(C)^{\rm op}} \; \underline{\rm Hom}_\mathcal{A}(C, Y),$$
where $C=\tau^{-1}K$.

We obtain a commutative triangle via the Auslander-Reiten duality.

\begin{prop}\label{prop:comm}
Let $\mathcal{A}$ be a Hom-finite $k$-linear abelian category having Auslander-Reiten duality, and let $K, Y$ be objects. Set $C=\tau^{-1} K$. Then the following triangle is commutative.
\[\xymatrix @R=1.5pc @C=0.8pc{
&  {{\rm sub}_{\Gamma(C)^{\rm op}} \; \underline{\rm Hom}_\mathcal{A}(C, Y)}  \\
[\longrightarrow Y\rangle_{\rm epi}  \ar[ru]^{\eta_{C,Y}}  \ar[rr]^{\delta_{K, Y}}&& {\rm sub}_{\Gamma(K)}\; {\rm Ext}^1_\mathcal{A}(Y, K) \ar[lu]_{\gamma_{K, Y}}
}\]
\end{prop}

\begin{proof}
Let $\alpha\colon X\rightarrow Y$ be an epimorphism. Consider the exact sequence $\xi_\alpha\colon 0\rightarrow {\rm Ker}\; \alpha\rightarrow X\rightarrow Y\rightarrow 0$. The following sequence is exact
$$\underline{\rm Hom}_\mathcal{A}(C, X)\xrightarrow{\underline{\rm Hom}_\mathcal{A}(C, \alpha)} \underline{\rm Hom}_\mathcal{A}(C, Y)\xrightarrow{c(C, \xi_\alpha)}  {\rm Ext}^1_\mathcal{A}(C, {\rm Ker}\; \alpha).$$
We have $\eta_{C, Y}([\alpha\rangle)={\rm Im}\;\underline{\rm Hom}_\mathcal{A}(C, \alpha)={\rm Ker}\; c(C, \xi_\alpha)$. Recall that $\delta_{K, Y}([\alpha\rangle)={\rm Im}\; c(\xi_\alpha, K)$. Then $\gamma_{K, Y}({\rm Im}\; c(\xi_\alpha, K))={\rm Ker}\; c(C, \xi_\alpha)$ by (\ref{equ:exact}) and (\ref{equ:compat}). Then we are done with $\eta_{C, Y}=\gamma_{K, Y}\circ \delta_{K, Y}$.
\end{proof}

The following important characterization of right determined epimorphisms in the module category over an artin algebra is contained in \cite[X.2]{ARS} (also see \cite{Rin1}), while the corresponding result in an abelian category with Serre duality is partly contained in \cite[Remark 3.5]{ChL}.

\begin{prop}\label{prop:C-deter}
Let $\mathcal{A}$ be a Hom-finite $k$-linear abelian category having Auslander-Reiten duality. Let $\alpha\colon X\rightarrow Y$ be an epimorphism with $K={\rm Ker}\; \alpha$. Then $\alpha$ is right $\tau^{-1} K$-determined. Moreover, if $\alpha$ is right minimal, $\alpha$ is right $C$-determined for some object $C$ if and only if $\tau^{-1} K\in {\rm add}\; C$. Consequently, we have \begin{align*}
_K[\longrightarrow Y\rangle_{\rm epi}={^{\tau^{-1}K}[\longrightarrow Y\rangle_{\rm epi}}.
\end{align*}
\end{prop}

Here, we use the following convention: in considering the statement ``$\tau^{-1} K\in {\rm add}\; C$" in $\mathcal{A}$, we mean that all the non-projective direct summands of $\tau^{-1}K$ lie in ${\rm add}\; C$.

\begin{proof}
For the first statement, we consider a morphism $t\colon T\rightarrow Y$ such that any morphism  $\phi\colon \tau^{-1} K\rightarrow T$ satisfies that $t\circ \phi$ factors through $\alpha$. To show that $t$ factors through $\alpha$ is equivalent to proving that the element $[\xi.t]$ in ${\rm Ext}^1_\mathcal{A}(T, K)$ is zero, where $\xi\colon 0\rightarrow K\rightarrow X\stackrel{\alpha}\rightarrow Y\rightarrow 0$ is the exact sequence corresponding to $\alpha$ and $\xi. t$ denotes the pullback of $\xi$ along $t$. By the isomorphism (\ref{equ:AR2}), it amounts to proving that the element $(D\Psi_{T, K})^{-1}([\xi.t])\in D\underline{\rm Hom}_\mathcal{A}(\tau^{-1}K, T)$ is zero.

 We assume that this is not the case. Then there exists a morphism $\phi\colon \tau^{-1}K\rightarrow T$ such that $(D\Psi_{T, K})^{-1}([ \xi.t])(\phi)\neq 0$. By the naturalness of $\Psi$ on the first variable, we compute that
\begin{align*}
(D\Psi_{T, K})^{-1}([\xi.t])(\phi)=(D\Psi_{\tau^{-1} K, K})^{-1}([\xi.(t\circ \phi)])({\rm Id}_{\tau^{-1}K}).
\end{align*}
But $t\circ \phi$ factors through $\alpha$ and thus $[\xi.(t\circ \phi)]=0$. This is a contradiction!

It remains to prove the ``only if" part of the second statement, since the ``if" part is trivial by the first statement. We claim that each indecomposable direct summand $K'$ of $K$ satisfies $\tau^{-1} K'\in {\rm add}\; C$.

For the claim, we observe that the composite inclusion $K'\stackrel{i}\rightarrow K\rightarrow X$ is not split, since $\alpha$ is right minimal; in particular, the object $K'$ is not injective. Consider the almost split sequence $\xi'\colon 0\rightarrow K'\stackrel{f}\rightarrow E\stackrel{g}\rightarrow \tau^{-1} K'\rightarrow 0$ starting at $K'$; see \cite[Section 4]{LZ}. Then the inclusion $K'\rightarrow X$ factors through $f$. We have the following commutative diagram
\[\xymatrix{
\xi'\colon \quad  0\ar[r] & K' \ar[d]^i \ar[r]^f & E\ar[r]^-g \ar@{.>}[d] & \tau^{-1} K'\ar[r]\ar@{.>}[d]^-t & 0\\
\xi\colon \quad  0 \ar[r] & K\ar[r] & X \ar[r]^\alpha & Y \ar[r] & 0.
}\]
In particular, we have $[i.\xi']=[\xi.t]$. We assume on the contrary that $\tau^{-1}K'$ does not lie in ${\rm add}\; C$. Then any morphism $\phi\colon C\rightarrow \tau^{-1}K'$ is not split epic, and thus factors through $g$. It follows that $t\circ \phi$ factor through $\alpha$. By the assumption that $\alpha$ is right $C$-determined, we infer that $t$ factors through $\alpha$, and thus the pullback $\xi.t$ of $\xi$ along $t$ is a split exact sequence. Then the pushout $i.\xi'$ of $\xi'$ along $i$ is split, and thus the split monomorphism $i$ factors through $f$. This is impossible, since $f$ is not split mono.
\end{proof}

We summarize the results in the following commutative bijection triangle, which involves the restricted Auslander bijections, universal extensions and the Auslander-Reiten duality.

\begin{thm}\label{thm:1}
Let $\mathcal{A}$ be a Hom-finite $k$-linear abelian category having Auslander-Reiten duality, and let $K, Y$ be objects. Then the following bijection triangle is commutative.
\begin{align}\label{equ:bij-tri}
\xymatrix @R=2pc @C=0.11pc{
&  {{\rm sub}_{\Gamma(\tau^{-1}K)^{\rm op}} \; \underline{\rm Hom}_\mathcal{A}(\tau^{-1}K, Y)}  \\
{_K[\longrightarrow Y\rangle_{\rm epi}}={^{\tau^{-1}K}[\longrightarrow Y\rangle_{\rm epi}}  \ar[ru]^{\eta_{\tau^{-1}K,Y}}  \ar[rr]^{\delta_{K, Y}}&& {\rm sub}_{\Gamma(K)}\; {\rm Ext}^1_\mathcal{A}(Y, K) \ar[lu]_{\gamma_{K, Y}}
}\end{align}
In particular, we have the restricted Auslander bijection at $Y$ relative to $\tau^{-1}K$
\begin{align}\label{equ:rabijec}
\eta_{\tau^{-1}K, Y}\colon {^{\tau^{-1}K}[\longrightarrow Y\rangle_{\rm epi}} \longrightarrow {{\rm sub}_{\Gamma(\tau^{-1}K)^{\rm op}} \; \underline{\rm Hom}_\mathcal{A}(\tau^{-1}K, Y)},
\end{align}
which is an isomorphism of posets.
\end{thm}

\begin{proof}
The equality at the left corner is contained in Proposition \ref{prop:C-deter}, while the commutativity is proved in Proposition \ref{prop:comm}. Since both the bijections $\delta_{K, Y}$ and $\gamma_{K, Y}$ are anti-isomorphisms of posets, it follows that $\eta_{\tau^{-1}K, Y}$ is an isomorphism of posets.
\end{proof}

\section{Consequences of the bijection triangle}

In this section, we draw some consequences of the bijection triangle (\ref{equ:bij-tri}). Universal extensions play a fundamental role in the restricted Auslander bijections. We verify a conjecture by Ringel in \cite{Rin2}; see Proposition \ref{prop:Ringel} and Corollary \ref{cor:Ringel}.

Throughout $k$ is a commutative artinian ring and $\mathcal{A}$ is a Hom-finite $k$-linear abelian category having Auslander-Reiten duality (\ref{equ:AR}). We recall that any morphism $\alpha\colon X\rightarrow Y$ is right equivalent to a right minimal morphism $\alpha'\colon X'\rightarrow Y$, in other words, $[\alpha\rangle=[\alpha'\rangle$ in $[\longrightarrow Y\rangle$; consult Corollary \ref{cor:rerm}. In what follows, we mainly consider right minimal epimorphisms.

Let $K, Y$ be objects in $\mathcal{A}$. Set $C=\tau^{-1}K$. We recall the restricted Auslander bijection (\ref{equ:rabijec}) at $Y$ relative to $C$
\begin{align*}
\eta_{C, Y}\colon {^C[\longrightarrow Y\rangle_{\rm epi}} \longrightarrow {{\rm sub}_{\Gamma(C)^{\rm op}} \; \underline{\rm Hom}_\mathcal{A}(C, Y)}.
\end{align*}

 The following result implies that universal extensions arise naturally in the restricted Auslander bijection; a special case is already  pointed out in \cite[Proposition 9.3]{Rin2}. Recall that universal extensions in $\mathcal{A}$ exist by Proposition \ref{prop:uni-ext}(3).

\begin{prop}
Let $\alpha\colon X\rightarrow Y$ be a right minimal epimorphism. Then the following statements are equivalent:
\begin{enumerate}
\item the morphism $\alpha$ is right $C$-determined and $\eta_{C, Y}([\alpha\rangle)=0$;
\item  ${\rm Ker}\; \alpha\in {\rm add}\; K$ and ${\rm Im}\; {\rm Hom}_\mathcal{A}(C, \alpha)=\mathcal{P}(C, Y)$;
\item the corresponding exact sequence $\xi_\alpha\colon 0\rightarrow {\rm Ker}\; \alpha\rightarrow X\stackrel{\alpha}\rightarrow Y\rightarrow 0$ is a universal extension of $Y$ by $K$.
    \end{enumerate}
\end{prop}

 \begin{proof}
 Recall that $\mathcal{P}(C, Y)\subseteq {\rm Im}\; {\rm Hom}_\mathcal{A}(C, \alpha)$, since each projectively trivial morphism factors through any epimorphism. Hence, $\eta_{C, Y}([\alpha\rangle)=0$ if and only if ${\rm Im}\; {\rm Hom}_\mathcal{A}(C, \alpha)=\mathcal{P}(C, Y)$. By Proposition \ref{prop:C-deter}, $\alpha$ is right $C$-determined if and only if ${\rm Ker}\; \alpha\in {\rm add}\; K$. Then ``$(1)\Leftrightarrow(2)$" follows.

 We observe that $\gamma_{K, Y} ({\rm Ext}_\mathcal{A}^1(Y, K))=0$. It follows from the bijection triangle (\ref{equ:bij-tri}) that  $\eta_{C, Y}([\alpha\rangle)=0$ if and only if $\delta_{K, Y}([\alpha\rangle)={\rm Ext}_\mathcal{A}^1(Y, K)$, which is further equivalent to that $\xi_\alpha$ is a semi-universal extension of $Y$ by $K$. Then  ``$(1)\Leftrightarrow(3)$" follows from Proposition \ref{prop:uni-ext}(1).
 \end{proof}

Following \cite[Section 11]{Rin2}, we say that an object $X$ is \emph{present} in ${^{C}[\longrightarrow Y\rangle_{\rm epi}}$ provided that there exists a right minimal epimorphism $X\rightarrow Y$ that is right $C$-determined. These objects are somehow controlled by universal extensions.  The following result is in spirit close to \cite[Proposition 11.1]{Rin2}.

\begin{prop}
Fix a universal extension $0\rightarrow \bar{K}\rightarrow \bar{X}\rightarrow Y\rightarrow 0$ of $Y$ by $K$. Then there exists a natural number $n=n(K, Y)$ with the following property: any object $X$ present in ${^C[\longrightarrow Y\rangle_{\rm epi}}$ fits into an exact sequence
\begin{align*}
0\longrightarrow \bar{K}\longrightarrow \bar{X}\oplus K'\longrightarrow X\longrightarrow 0
\end{align*}
for a direct summand $K'$ of $K^n$. In particular, $X$ is a quotient object of $\bar{X}\oplus K^n$.
\end{prop}

Here, $K^n$ denotes the direct sum of $n$ copies of $K$.

\begin{proof}
We observe that there exists a natural number $n$ such that for any $\Gamma(K)$-submodule $L$ of ${\rm Ext}^1_\mathcal{A}(Y, K)$ there is an epimorphism ${\rm Hom}_\mathcal{A}(K^n, K)\rightarrow L$ of $\Gamma(K)$-modules. For example, we may take $n$ to be the length of ${\rm Ext}_\mathcal{A}^1(Y, K)$ as a $k$-module.

 Take any right minimal epimorphism $\alpha\colon X\rightarrow Y$ that is right $C$-determined; in particular, ${\rm Ker}\; \alpha\in {\rm add}\; K$. Then the corresponding exact sequence $\xi_\alpha \colon 0\rightarrow {\rm Ker}\; \alpha\rightarrow X\stackrel{\alpha}\rightarrow Y\rightarrow 0$ is a universal $L$-extension for some $L$; see Proposition \ref{prop:uni-ext}(1). Thus the connecting map $c(\xi_\alpha, K)\colon {\rm Hom}_\mathcal{A}({\rm Ker}\; \alpha, K)\rightarrow L$ is a projective cover. It follows that ${\rm Ker}\;\alpha$ is a direct summand of $K^n$. Using the universal property of the (semi-)universal extension, we have the following commutative diagram; compare  Lemma \ref{lem:factor}(2) and (3).
\[\xymatrix{
0 \ar[r] & \bar{K} \ar@{.>}[d] \ar[r] & \bar{X} \ar@{.>}[d] \ar[r] & Y\ar[r] \ar@{=}[d] & 0\\
0 \ar[r] & {\rm Ker}\; \alpha \ar[r] & X \ar[r]^\alpha & Y \ar[r] & 0
}\]
The left commutative square yields the required exact sequence.
\end{proof}

Recall that $K, Y$ are objects in $\mathcal{A}$ and $C=\tau^{-1}K$.  Consider $K'\in {\rm add}\; K$ and $C'=\tau^{-1}K'$. By Proposition \ref{prop:C-deter} the following map is well defined
$$\kappa\colon{\rm Ext}_\mathcal{A}^1(Y, K')\longrightarrow {^C[\longrightarrow Y\rangle_{\rm epi}}$$
such that $\kappa([\xi])=[\alpha\rangle$ for any exact sequence $\xi\colon 0\rightarrow K'\rightarrow X\stackrel{\alpha}\rightarrow Y\rightarrow 0$.

The following consideration extends slightly the one in \cite[Section 10]{Rin2}. We consider the following composite
$$ F \colon D\underline{\rm Hom}_\mathcal{A}(C', Y)\stackrel{D\Psi_{Y, K'}}\longrightarrow {\rm Ext}_\mathcal{A}^1(Y, K')\stackrel{\kappa}\longrightarrow {^C[\longrightarrow Y\rangle_{\rm epi}}\stackrel{\eta_{C, Y}}\longrightarrow {\rm sub}_{\Gamma(C)^{\rm op}}\; \underline{\rm Hom}_\mathcal{A}(C, Y).$$

We have the following characterization of the map $F$. Here, we recall that elements of $D\underline{\rm Hom}_\mathcal{A}(C', Y)$ are $k$-linear maps $\underline{\rm Hom}_\mathcal{A}(C', Y)\rightarrow E$, where $E$ is the minimal injective cogenerator of $k$.

\begin{prop}\label{prop:Ringel}
For any $k$-linear map $\theta\in D\underline{\rm Hom}_\mathcal{A}(C', Y)$, we have
$$F(\theta)=\{f\in  \underline{\rm Hom}_\mathcal{A}(C, Y)\; |\; \theta(f\circ g)=0 \mbox{ for any morphism } g\colon C'\rightarrow C\}.$$
\end{prop}

\begin{proof}
Assume that $(D\Psi_{Y, K'})(\theta)=[\xi]$, where $\xi\colon 0\rightarrow K'\rightarrow X\stackrel{\alpha}\rightarrow Y\rightarrow 0$ is an exact sequence. Consider the composite
$$\delta_{K, Y}\circ \kappa\colon {\rm Ext}_\mathcal{A}^1(Y, K)\longrightarrow {\rm sub}_{\Gamma(K)}\; {\rm Ext}_\mathcal{A}^1(Y, K).$$ We claim that $\delta_{K, Y}\circ \kappa ([\xi])={\rm Im}\; c(\xi, K)$. Indeed, $\kappa([\xi])=[\alpha\rangle$, and thus $\delta_{K, Y}([\alpha\rangle)={\rm Im}\; c(\xi_\alpha, K)={\rm Im}\; c(\xi, K)$.

By the bijection triangle (\ref{equ:bij-tri}), we have $F(\theta)=\gamma_{K, Y}({\rm Im}\; c (\xi, K))=H$, which we will compute.  Recall from Lemma \ref{lem:AR-bij} the following exact sequence
\begin{align*}
0\longrightarrow H\stackrel{\rm inc}\longrightarrow \underline{\rm Hom}_\mathcal{A}(C, Y)\stackrel{\psi}\longrightarrow D{\rm Im}\; c(\xi, K)\longrightarrow 0,
\end{align*}
where ``${\rm inc}$" denotes the inclusion and $\psi=D({\rm inc}')\circ (\Psi_{Y, K})^{-1}$. Here, ${\rm inc}'\colon {\rm Im}\; c(\xi, K) \rightarrow  {\rm Ext}_\mathcal{A}^1(Y, K)$ is the inclusion.

Consider the following $k$-module
\begin{align*}S=\{\theta'\in D\underline{\rm Hom}_\mathcal{A}(C, Y)\; |\; \mbox{there is a  morphism }  & g\colon C'\rightarrow C \mbox{ with } \\ \theta'(f)&=\theta(f\circ g) \mbox{ for each } f\colon C\rightarrow Y\},\end{align*}
and the corresponding inclusion  ${\rm inc}''\colon S\rightarrow D\underline{\rm Hom}_\mathcal{A}(C, Y)$.

Recall that ${\rm Im}\; c(\xi, K)=\{[u.\xi]\in {\rm Ext}_\mathcal{A}^1(Y, K)\; |\; u\colon K'\rightarrow K\}$. We apply the naturalness of $\Psi$ on the second variable and the equivalence $\tau^{-1}$. Then the isomorphism $D\Psi_{Y, K}$ identifies $S$ with ${\rm Im}\; c(\xi, K)$.  Hence, we have the following commutative diagram
\[\xymatrix{
\underline{\rm Hom}_\mathcal{A}(C, Y)\ar[d]^{\Psi_{Y, K}^{-1}} \ar[rr]^{D{\rm inc}''}&& DS \ar[d]^{\Psi_{K, Y}^{-1}}\\
D{\rm Ext}^1_\mathcal{A}(Y, K)\ar[rr]^{D{\rm inc}'} && D{\rm Im}\; c(\xi, K).
}\]
It follows that  $H$ coincides with the kernel of  $D{\rm inc}'' \colon \underline{\rm Hom}_\mathcal{A}(C, Y) \rightarrow DS$. We observe that $D{\rm inc}''$  sends $f\colon C\rightarrow  Y$ to a $k$-linear map on $S$, which sends $\theta'\in S$ to $\theta'(f)$. Hence, $f$ lies in $H$ if and only if $\theta'(f)=0$ for each $\theta'\in S$, if and only if $\theta(f\circ g)=0$ for each morphism $g \colon C'\rightarrow C$. This is the required  characterization of $F(\theta)=H$.
\end{proof}

We have the following immediate consequence, which  is conjectured by Ringel in \cite[Section 10]{Rin2}.

\begin{cor}\label{cor:Ringel}
Consider the case that $K'=K$ and thus $C'=C$. Then for any $k$-linear map $\theta\in D\underline{\rm Hom}_\mathcal{A}(C, Y)$, we have
$$F(\theta)=\{f\in  \underline{\rm Hom}_\mathcal{A}(C, Y)\; |\; \theta(f\circ g)=0 \mbox{ for any } g\in \Gamma(C)\},$$
which is the largest $\Gamma(C)^{\rm op}$-submodule  of $\underline{\rm Hom}_\mathcal{A}(C, Y)$ contained in ${\rm Ker}\; \theta$.
\end{cor}

\begin{proof}
The first statement is a special case of Proposition \ref{prop:Ringel}, and the second statement is direct to verify.
\end{proof}

\vskip 5pt

\noindent{\bf Acknowledgements} \quad The author is very grateful to the two referees for many helpful suggestions, and thanks Pengjie Jiao for helpful discussion. This work is supported by National Natural Science Foundation of China (No.s 11522113 and 11671245), and the Fundamental Research Funds for the Central Universities.

\bibliography{}

\begin{thebibliography}{999}

\bibitem{ARS}{\sc M. Auslander, I. Reiten, and S.O. Smal{\o},}
Representation Theory of Artin Algebras, Cambridge Studies in Adv.
Math. {\bf 36}, Cambridge Univ. Press, Cambridge, 1995.


\bibitem{Aus1} {\sc M. Auslander}, {\em Functors and morphisms determined by objects}, in: Representation theory of algebras (Proc. Conf. Temple Univ., Philadelphia, Pa., 1976), 1--244, Lecture Notes in Pure Appl. Math. {\bf 37}, Dekker, New York, 1978.

\bibitem{Aus2} {\sc M. Auslander}, {\em Applications of morphisms determined by modules}, in: Representation theory of algebras (Proc. Conf. Temple Univ., Philadelphia, Pa., 1976), 245--327, Lecture Notes in Pure Appl. Math. {\bf 37}, Dekker, New York, 1978.

\bibitem{AR} {\sc M. Auslander and I. Reiten}, {\em Stable equivalence of dualizing $R$-varieties}, Adv. Math. {\bf 12} (1974), 306--366.

\bibitem{BLP} {\sc R. Bautista, S. Liu, and C. Paquette}, {\em Representation theory of strongly locally finite
quivers}, Proc. Lond. Math. Soc. (3) {\bf 106} (2013), 97--162.


\bibitem{ChL} {\sc X.W. Chen and J. Le}, {\em A note on morphisms determined by objects}, J. Algebra {\bf 428} (2015), 138--148.

\bibitem{Ji} {\sc P. Jiao}, {\em The generalized Auslander-Reiten duality on an exact category}, arXiv: 1609.07732v2.

\bibitem{Kr} {\sc H. Krause}, {\em Morphisms determined by objects in triangulated categories}, In ``Algebras, Quivers and Representations'',  The Abel Symposium 2011, 195--207, Abel Symposia {\bf 8}, Springer, 2013.

\bibitem{LZ} {\sc H. Lenzing and R. Zuazua}, {\em Auslander-Reiten duality for abelian categories}, Bol. Soc. Mat. Mexicana (3) {\bf 10} (2004), 169--177.

\bibitem{Rin1} {\sc C.M. Ringel}, {\em Morphisms determined by objects: The case of modules over artin algebras}, Illinois J. Math. {\bf 56} (2012), no. 3, 981--1000.

\bibitem{Rin2} {\sc C.M. Ringel}, {\em The Auslander bijections: How morphisms are determined by objects}, Bull. Math. Sci. {\bf 3} (2013), no.3,  409--484.

\bibitem{Rin3} {\sc C.M. Ringel}, {\em Quiver Grassmannians and Auslander varieties for wild algebras}, J. Algebra {\bf 402} (2014), 351--357

\end{thebibliography}

\vskip 10pt

 {\footnotesize \noindent Xiao-Wu Chen \\
 Key Laboratory of Wu Wen-Tsun Mathematics, Chinese Academy of Sciences,\\
School of Mathematical Sciences, University of Science and Technology of China,\\
No. 96 Jinzhai Road, Hefei, 230026, Anhui, P.R. China.\\
E-mail: xwchen@mail.ustc.edu.cn, URL: http://home.ustc.edu.cn/$^\sim$xwchen}

\end{document}